\documentclass[12pt]{article}

\makeatletter
\newfont{\footsc}{cmcsc10 at 8truept}
\newfont{\footbf}{cmbx10 at 8truept}
\newfont{\footrm}{cmr10 at 10truept}
\makeatother
\usepackage{amsmath}
\usepackage{amssymb}
\usepackage{exscale}
\usepackage{amsthm}
\usepackage{epsfig}
\usepackage{verbatim}
\usepackage{setspace}
\usepackage{bbm}

\theoremstyle{plain}
\newtheorem{theorem}{Theorem}[section]
\newtheorem{proposition}[theorem]{Proposition}
\newtheorem{lemma}[theorem]{Lemma}
\newtheorem{corollary}[theorem]{Corollary}

\theoremstyle{definition}

\newtheorem{remark}[theorem]{Remark}

\DeclareMathOperator{\dist}{dist}

\DeclareMathOperator{\Prob}{{\bf P}}

\DeclareMathOperator{\modulo}{mod}
\DeclareMathOperator{\Var}{Var}

\newcommand{\R}{{\mathbb{R}}}

\def\bs0{\bf 0}

\title{Note on Pairwise Negative Dependence of Randomized
Rank-1 Lattices}

\author{Marcin Wnuk\thanks{Mathematisches Seminar, Christian-Albrechts-Universit\"at zu Kiel,
Germany ({\tt wnuk@math.uni-kiel.de}).}
\and Michael Gnewuch\thanks{Institut f\"ur Mathematik, Universit\"at Osnabr\"uck,
Germany ({\tt michael.gnewuch@uni-osnabrueck.de}).}
}

\begin{document}

\maketitle
\vskip 1pc


\begin{abstract}
In her recent paper [Negative dependence, scrambled nets, and variance bounds. Math. Oper.~Res. 43 (2018), 228-251]  Christiane Lemieux studied a framework to analyze the dependence structure of sampling schemes. The main goal of the framework is to determine conditions under which the negative dependence structure of a sampling scheme yields estimators with reduced variance compared to Monte Carlo estimators. For instance, she was able to show that in dimension $d=2$ scrambled $(0,m,d)$-nets lead to randomized quasi-Monte Carlo estimators with variance no larger than the variance of  Monte Carlo estimators for functions monotone in each variable. Her result relies on a pairwise negative dependence property that is, in particular, satisfied by $(0,m,2)$-nets.
In this note we establish that the same result holds true in arbitrary dimension $d$ for a type of randomized lattice point sets that we call randomly shifted and jittered rank-$1$ lattices.
We show that the details of the randomization are crucial and that already small modifications may destroy the pairwise negative dependence property.
\end{abstract}

\section{Introduction}

Monte Carlo (MC) sampling is a frequently used method in stochastic simulation and in multivariate numerical integration.
Let $p_1, \ldots, p_N$ be independent random points, uniformly distributed on
the $d$-dimensional unit cube $[0,1]^d$. For a given square-integrable function $f: [0,1]^d \to \R$ we consider the MC estimator (or quadrature)
\begin{equation}\label{MC_estimator}
\mu^{MC}(f) = \frac{1}{N} \sum_{i=1}^N f(p_i)
\end{equation}
to estimate the expected value (or integral)
\begin{equation*}
I(f) = \int_{[0,1]^d} f(u) \, du.
\end{equation*}

An advantage of the MC estimator is that already under the very mild assumption on $f$ to be square integrable, it  converges to $I(f)$ for $N\to \infty$ with convergence rate $1/2$.
Although the convergence rate is far away from beeing impressive, it has the invaluable advantage that it does not depend on the numbers of variables $d$ (or, to be precise, at least not directly, see, e.g., \cite{SW04}).

Some dependent sampling schemes are in many respects superior to MC sampling.
An example are randomized quasi-Monte Carlo (RQMC) methods.
They ensure, for instance, faster convergence rates  for numerical integration of sufficiently smooth functions, they exhibit much smaller  asymptotic discrepancy measures, and their sample points have more evenly distributed lower dimensional projections (see, e.g., \cite{DKS13, DP10, Lem09}).

It would be desirable to have dependent sampling schemes sharing some of these favorable properties, and that are, with respect to other objectives, at least as good as MC sampling schemes.

Recently, there has been some research in this direction. In \cite{GH18, Heb12} the authors showed that a specific negative dependence property of RQMC point sets guarantees that they satisfy the same pre-asymptotic probabilistic discrepancy bounds (with explicitly  revealed dependence on the number of points $N$ as well as on the dimension $d$) as MC points.
Here the negative dependence property relies on
the common distribution of all sample points.
Related and further results in this direction can be found in \cite{DDG18, WGH19}.

In \cite{Lem17} Christiane Lemieux showed that  another negative dependence property of RQMC points ensures, in particular, that the variance of the corresponding RQMC estimator for functions that are monotone with respect to each variable is never larger than the variance of the MC estimator.
The property relies solely on the distribution of single points and on the common distribution of pairs of points. For a precise definition of this property we need to introduce some terminology.

Let $N,d \in \mathbb{N}.$ We call a randomized point set $\mathcal{P} = (p_j)_{j = 1}^N$ a \emph{sampling scheme} if every single $p \in \mathcal{P}$ is distributed uniformly in $[0,1)^d$ and the vector $(p_1, \ldots, p_N)$ is exchangeable, meaning that for any permutation $\pi$ of $[N]$ it holds that the law of $(p_1, \ldots, p_N)$ is the same as the law of $(p_{\pi(1)}, \ldots, p_{\pi(N)}).$ The assumption of exchangeability is only of technical nature and poses no additional constraints on a sampling scheme seen as a point set.

 Let us denote by $\mathcal{C}_1^d$ the set of all \emph{boxes anchored at} $1,$ i.e.
$$\mathcal{C}_1^d = \{[x,1) \, | \, x \in [0,1)^d \}.$$
A sampling scheme $\mathcal{P} = (p_j)_{j = 1}^N$ is \emph{pairwise negatively dependent} if for every $Q,R \in \mathcal{C}_1^d$ it holds
\begin{equation*}
\Prob(p_1 \in Q, p_2 \in R) \leq \Prob(p_1 \in Q)\Prob(p_2 \in R).
\end{equation*}

 In \cite{Lem17} this type of sampling is called \emph{negatively upper orthant dependent (NUOD) sampling scheme}. Note by the way that in the corresponding definition \cite[p.230]{Lem17} below formula $(5),$ there is a typing error: it should be $T({\bf{u}}, {\bf{v}},\tilde{P}_n) \leq \prod_{l = 1}^s (1-u_l)(1-v_l)$ and not $T({\bf{u}}, {\bf{v}},\tilde{P}_n) \geq \prod_{l = 1}^s (1-u_l)(1-v_l).$   Combining Proposition 3 and Remark 8 from \cite{Lem17} one sees that pairwise negatively dependent sampling schemes  lead to estimators of integrals with variance not larger then the variance of MC estimators for bounded quasimonotone functions. For the definition of quasimonotonicity we refer to \cite[Definition 3]{Lem17}.

Let $p_j = (p_j^{(1)},\ldots, p_j^{(d)}), j = 1, \ldots, N.$ If for every $i = 1, \ldots, d,$ the distribution of $(p_j^{(i)})_{j = 1}^N$ given the distribution of $(p_j^{(k)})_{j = 1}^N, k = 1, \ldots, i-1,$ is pairwise negatively dependent we say that the sampling scheme $(p_j)_{j = 1}^N$ is \emph{conditionally negatively quadrant dependent} (conditionally NQD). Note that this holds in particular if $(p_1^{(i)}, p_2^{(i)})_{i = 1}^d$ are independent and for every $i = 1, \ldots, d,$  and every $q,r \in [0,1)$ we have
$$\Prob(p_1^{(i)} \in [q,1), p_2^{(i)} \in [r,1)) \leq \Prob(p_1 \in [q,1)) \Prob(p_2 \in [r,1)),$$
in which case we talk of a \emph{coordinatewise independent NQD sampling scheme}. Christiane Lemieux showed in \cite[Corollary 2]{Lem17} that conditionally NQD sampling schemes provide RQMC estimators of integrals with variance no bigger then the variance of the MC estimator if the integrand is monotone in each coordinate. We remark here that the definition of the NQD property given in \cite{Lem17} differs slightly from the one given by us, but both definitions are easily seen to be equivalent.

Moreover, Lemieux was able to establish that scrambled $(0,m,2)-$nets are actually pairwise negatively dependent and conditionally NQD sampling schemes, see \cite[Corollary 1, proof of Proposition 12]{Lem17}. In this note we prove that the same is true for suitably randomized rank-1 lattices in arbitrary dimension. As we will demonstrate, the specific way of randomization turns out to be important.


\section{Simple Stratified Sampling and Multidimensional Extensions}
 One of the most frequently used one-dimensional sampling schemes is the \emph{simple stratified sampling}, defined in the following way: let $\pi$ be a uniformly chosen permutation of $\{1, \ldots, N\}$ and let $(R_j)_{j = 1}^N$ be independent random variables distributed uniformly on $(0,1].$ Moreover, $(R_j)_j$ is independent of $\pi.$ We put
$$p_j := \frac{\pi(j) - R_j}{N}, \hspace{3ex} j = 1, \ldots, N. $$
Effectively, one is considering the partition $I_j := [\tfrac{j-1}{N}, \tfrac{j}{N}), j = 1, \ldots, N,$ of the unit interval and in every element of the partition putting one point, independently of all the other points.

Formally, at least for $N = b^m,$ where $b \geq 2, m \geq 1$ are integers, pairwise negative dependence of stratified sampling may be deduced from $\cite{Lem17}.$ Still, the proof presented there is rather involved, since it shall easily generalize to a proof of pairwise negative dependence of scrambled $(0,m,2)-$nets. For this reason we present here an easy argument yielding pairwise negative dependence of simple stratified sampling. Similar results may be also found in the literature, see e.g. Lemma $3.4.$ in \cite{GH18}.
\begin{lemma}\label{SSSpND}
 Simple stratified sampling $\mathcal{P} = (p_j)_{j = 1}^N$ is pairwise negatively dependent.
\end{lemma}
\begin{proof}
Let $Q = [q,1), R = [r,1)$ be two boxes anchored at $1.$ Without loss of generality we may assume $R \subset Q.$ We aim at showing
$$ \Prob(p_1 \in Q | p_2 \in R) \leq \lambda(Q).$$
Let $\eta, \rho$ be such that $q \in I_{\eta}, r \in I_{\rho}.$ Define $\epsilon_q := \eta - Nq, \epsilon_r := \rho - Nr.$ We are considering two cases. In the first case $\eta < \rho.$ Then it follows
\begin{align*}
& \Prob(p_1 \in Q | p_2 \in R) = \Prob(p_1 \in Q, p_2 \in I_{\rho} | p_2 \in R) + \sum_{k = \rho + 1}^N \Prob(p_1 \in Q, p_2 \in I_k | p_2 \in R)
\\
& = \frac{N - \eta - 1 + \epsilon_q}{N}\frac{\epsilon_r}{N \lambda(R)} +   \frac{N - \eta - 1 + \epsilon_q}{N} \frac{N - \rho}{N \lambda(R)}
\\
& = \frac{N - \eta - 1 + \epsilon_q}{N} < \frac{N - \eta + \epsilon_q}{N} = \lambda(Q).
\end{align*}
In the second case, $\eta = \rho$ and
\begin{align*}
&\Prob(p_1 \in Q | p_2 \in R) = \Prob(p_1 \in Q, p_2 \in I_{\eta} | p_2 \in R) + \sum_{k = \eta+1}^N \Prob(p_1 \in Q, p_2 \in I_k | p_2 \in R)
\\
& = \frac{N - \eta}{N} \frac{\epsilon_r}{N \lambda(R)} +  \frac{N - \eta - 1 + \epsilon_q}{N} \frac{N - \eta}{N \lambda(R)} < \frac{N- \eta + \epsilon_q}{N} = \lambda(Q).
\end{align*}
\end{proof}

Multidimensional extensions of simple stratified sampling include
\begin{enumerate}
\item[a)] stratified sampling in $[0,1)^d$, where the $N$ strata are axis parallel boxes,
\item[b)] Latin hypercube sampling,
\item[c)] randomly shifted and jittered rank-1 lattice,
\item[c)] fully scrambled $(0,m,d)-$nets in base $b \in \mathbb{N}_{\geq 2}.$
\end{enumerate}
In this note we focus on the randomly shifted and jittered rank-1 lattices and on Latin hypercube sampling, see sections \ref{R1L} and \ref{LHS}, respectively. For information on scrambled $(0,m,d)-$nets we refer to, e.g. \cite{DP10,Lem09,Lem17,Owe95,Owe97}. A generalization of stratified sampling is discussed, e.g., in \cite{WGH19}.

\subsection{Randomly Shifted and Jittered Rank-1 Lattice}\label{R1L}
Let $N$ be prime. We denote $\mathbb{F} := \mathbb{F}_{N}:=\{0,1,\ldots, N-1\}$. Moreover, $\mathbb{F}^* := \mathbb{F} \setminus \{0\}.$ We also put $\widetilde{\mathbb{F}} := \frac{1}{N} \mathbb{F}$ and $\widetilde{\mathbb{F}}^* := \tfrac{1}{N} \mathbb{F}^*.$

A discrete subgroup $L$ of the $d-$dimensional torus $\mathbb{T}^d$ (where $\mathbb{T}^d = [0,1)^d,$ the addition of two elements of $\mathbb{T}^d$ and the multiplication with reals is to be taken componentwise modulo $1.$ ) is called a lattice. A set $(y_j)_{j = 1}^N$ is a rank-1 lattice if for some $g \in (\widetilde{\mathbb{F}}^*)^d$ it admits a representation
$$y_j = (j-1)g \,  \modulo 1, \quad  j = 1,\ldots, N.$$ In this case $g$ is called a generating vector of the lattice. Note that, in particular, a rank-1 lattice is a cyclic subgroup of the torus.

We remark that our definition differs from the usual one in that we allow only for generating vectors $g$ from $(\widetilde{\mathbb{F}}^*)^d$ and not from $\widetilde{\mathbb{F}}^d,$ which saves us from considering some degenerate cases.

We want now to define a sampling scheme based on rank-1 lattices. To this end let $(y_j)_{j = 1}^N$ be a rank-1 lattice with generating vector chosen uniformly at random from $(\widetilde{\mathbb{F}}^*)^d.$ Let $S$ be distributed uniformly on $\widetilde{\mathbb{F}}^d,$ $J_j, j = 1, \ldots, N$ be uniformly distributed on $[0, \frac{1}{N})^d$ and $\pi$ be a uniformly chosen permutation of $\{1,\ldots, N\}.$ Moreover, let all of the aforementioned random variables be independent. We put
$$p_{j}:= y_{\pi(j)} + S + J_j \text{ mod } 1, \hspace{3ex} j = 1, \ldots, N,$$
and call the sampling scheme $\mathcal{P} = (p_j)_{j = 1}^N$ a \emph{randomly shifted and jittered rank-1 lattice (RSJ rank-1 lattice) }.
Putting it in words: we first take a rank-1 lattice with a random generator and symmetrize it. Then we shift the lattice uniformly on the torus, where the shift has resolution $\frac{1}{N}$. In the last step we jitter every point independently of all the other points.

\subsection{Latin Hypercube Sampling}\label{LHS}
Let $(\pi_i)_{i = 1}^d$ be independent uniformly chosen permutations of $\{1,\ldots, N\},$ and $U^{(i)}_j, i = 1,\ldots, d, j = 1, \ldots, N$ be independent random variables distributed uniformly on $[0,1)$ and independent also of the permutations. A sampling scheme $(p_j)_{j=1}^N$ is called a \emph{Latin hypercube sampling} if the $i-$th coordinate of the $j-$th point $p_j^{(i)}$ is given by
\begin{equation}\label{LHS_def}
p_j^{(i)} = \frac{\pi_{i}(j) - U^{(i)}_j}{N}, \hspace{3ex} i = 1,\ldots, d, j = 1,\ldots, N.
\end{equation}
What one intuitively does is the following: one cuts $[0,1)^d$ into slices $(S_{k,j})_{j = 1}^N, k = 1, \ldots, d,$ given by
$$S_{k,j} = \prod_{j = 1}^{k-1}[0,1) \times [\tfrac{j-1}{N},\tfrac{j}{N}) \times \prod_{j = k+1}^d [0,1) $$
and places $N$ points in such a way that in every slice there is exactly one point.

Latin hypercube sampling was introduced in \cite{MBC79}. An earlier variant, known as lattice sampling, is due to \cite{Pat54}. There one simply substitutes the random variables $U^{(i)}_j$ in (\ref{LHS_def}) by constant values $\frac{1}{2}.$   The negative dependence properties of LHS were investigated in \cite{GH18}.

\section{Pairwise Negative Dependence of RSJ Rank-1 Lattice and LHS}
Our aim is now to show that RSJ rank-1 lattice is pairwise negatively dependent. In the course of the proof we will demonstrate that the bivariate copulas (in this case: cummulative distribution functions of a pair of points) of RSJ rank-1 lattice and LHS are actually the same, and then prove the result for LHS.

 Recall that for $d = 1$ RSJ rank-1 lattice and LHS are nothing else but simple stratified sampling, pairwise negative dependence is therefore settled by Lemma \ref{SSSpND}.

Now we may turn our attention to the more interesting multidimensional case. Let $\mathcal{P} = (p_j)_{j = 1}^N$ be a RSJ rank-1 lattice in $[0,1)^d, d \geq 2.$
Put
$$D = \{(f_1,f_2) \in \mathbb{F} \times \mathbb{F} \, | \, f_1 \neq f_2 \}.$$
A random variable $(p_1, p_2)$ having uniform distribution on pairs of distinct points from $\mathcal{P}$ may be generated in the following way: let $(m_1,m_2)$ be uniformly distributed on $D$ and let the random variables $g, S, J_1,\ldots, J_N$ be independent. The generating vector $g$ is distributed uniformly on  $(\widetilde{\mathbb{F}}^*)^d,$ the shift $S$ on $\widetilde{\mathbb{F}}^d.$ If $J_j, j = 1, \ldots, N,$ are all equal to $0,$ then we speak of the \emph{discrete model}, and if $J_j, j =1,\ldots, N,$ are distributed uniformly on $[0,\frac{1}{N})^d,$ we speak of the \emph{continuous model.} Even though our aim is to investigate the continuous model, the discrete model will turn out to be helpful to highlight the combinatorial nature of the problem. Finally we put for $j=1,2, \,\, i = 1,\ldots,d,$
$$p^{(i)}_{j} = (g^{(i)}m_j + S^{(i)}) \modulo 1 + J^{(i)}_j.$$

\begin{lemma}\label{JointUnifDistLemma}
In the discrete model for $(z_1,z_2) \in \tfrac{1}{N} D^d $ and  $(a,b) \in D$ it holds
$$\Prob(p_1 = z_1, p_2 = z_2 | m_1 = a, m_2 = b) = \frac{1}{(N(N-1))^d}.$$
In particular, in the discrete model
$$\Prob(p_1 = z_1, p_2 = z_2) = \frac{1}{(N(N-1))^d}.$$
\end{lemma}
\begin{proof}
In the first step we show that for any $i = 1,\ldots, d$
\begin{equation}\label{OneDimInd}
\Prob(p_1^{(i)} = z_1^{(i)}, p_2^{(i)} = z_2^{(i)} | m_1 = a, m_2 = b) = \frac{1}{N(N-1)}.
\end{equation}
It suffices to show that given $a, b, z_1^{(i)}, z_2^{(i)}$ the system of equations
\[
\left\{
\begin{array}{l}
z_1^{(i)} = \gamma a + \nu \hspace{3ex} \modulo 1 \\
z_2^{(i)} = \gamma b + \nu \hspace{3ex} \modulo 1
\end{array}
\right.
\]
has exactly one solution $\gamma \in \widetilde{\mathbb{F}}^*, \nu \in \widetilde{\mathbb{F}}.$ One solution is given by
\[
\left\{
\begin{array}{l}
\gamma = (z_1^{(i)} - z_2^{(i)})(a - b)^{-1} \modulo 1 \\
\nu = (z_1^{(i)} - \gamma a) \modulo 1
\end{array}
\right.
\]
and it is indeed unique, since the determinant of the associated matrix is $(a - b) \neq 0 \modulo N.$
Now we are ready to prove the claim of the theorem by induction on the dimension. Suppose the statement has already been proven for dimension $d$ and we want to prove it for dimension $(d+1).$ For any $(d+1)-$dimensional (possibly random) vector $W$ we denote by $\widetilde{W}$ the projection onto its first $d$ coordinates. We have
\begin{align*}
&\Prob(p_1 = z_1, p_2 = z_2 | m_1 = a, m_2 = b)
\\
& = \Prob(p_1^{(d+1)} = z_1^{(d+1)}, p_2^{(d+1)} = z_2^{(d+1)} | m_1 = a, m_2 = b, \widetilde{p}_1 = \widetilde{z}_1, \widetilde{p}_2 = \widetilde{z}_2)
\\
& \times \Prob(\widetilde{p}_1 = \widetilde{z}_1, \widetilde{p}_2 = \widetilde{z}_2| m_1 = a, m_2 = b).
\end{align*}
By induction assumption $\Prob(\widetilde{p}_1 = \widetilde{z}_1, \widetilde{p}_2 = \widetilde{z}_2| m_1 = a, m_2 = b) = \frac{1}{(N(N-1))^d},$
hence it suffices to show
\begin{align*}
& \Prob(p_1^{(d+1)} = z_1^{(d+1)}, p_2^{(d+1)} = z_2^{(d+1)} | m_1 = a, m_2 = b, \widetilde{p}_1 = \widetilde{z}_1, \widetilde{p}_2 = \widetilde{z}_2) = \frac{1}{N(N-1)}.
\end{align*}
Note now that conditioned on $\{m_1 = a, m_2 = b\},$ the events $\{p_1^{(d+1)} = z_1^{(d+1)}, p_2^{(d+1)} = z_2^{(d+1)}\}$ and $\{\widetilde{p}_1 = \widetilde{z}_1, \widetilde{p}_2 = \widetilde{z}_2 \}$ are independent and so by induction hypothesis we obtain

\begin{align*}
& \Prob(p_1^{(d+1)} = z_1^{(d+1)}, p_2^{(d+1)} = z_2^{(d+1)} | m_1 = a, m_2 = b, \tilde{p}_1 = \tilde{z}_1, \tilde{p}_2 = \tilde{z}_2)
\\
&= \Prob(p_1^{(d+1)} = z_1^{(d+1)}, p_2^{(d+1)} = z_2^{(d+1)} | m_1 = a, m_2 = b) = \frac{1}{N(N-1)}.
\end{align*}
This proves the first statement. The second statement follows immediately by the law of total probability.
\end{proof}

\begin{corollary}\label{JointUnifDistCor}
\begin{itemize}
\item[(i)]
The random variables
$$(p_1^{(i)}, p_2^{(i)}), \hspace{3ex} i=1, \ldots, d,$$
are independent and identically distributed in the discrete as well as in the continuous model.
\item[(ii)]
Randomly shifted and jittered rank-1 lattice has the same bivariate distributions as Latin hypercube sampling.
\end{itemize}
\end{corollary}
\begin{proof}
To prove (i) note that for given $(z_1,z_2) \in \tfrac{1}{N}D^d$ we obtain from Lemma \ref{JointUnifDistLemma} applied to the one-dimensional case
$$ \Prob(p_1^{(i)} = z_1^{(i)}, p_2^{(i)} = z_2^{(i)}) = \frac{1}{N(N-1)}, \hspace{3ex} i = 1,\ldots,d. $$
If $I \subset \{1,\ldots,d\},$ then by the same lemma applied to the $|I|-$dimensional case we see that
$$\Prob\left(\bigcap_{i \in I} \{p_1^{(i)} = z_1^{(i)}, p_2^{(i)} = z_2^{(i)} \}\right) = \frac{1}{(N(N-1))^{|I|}} = \prod_{i \in I} \Prob(p_1^{(i)} = z_1^{(i)}, p_2^{(i)} = z_2^{(i)}),$$
which yields the claim.

For (ii) let $(\widetilde{p}_j)_{j=1}^N$ be a LHS. Due to symmetrization $\widetilde{p}_1,\ldots, \widetilde{p}_N,$ are exchangeable and clearly we have that the random variables $(\widetilde{p}_1^{(i)},\widetilde{p}_2^{(i)}), i = 1, \ldots, d,$ are independent. Furthermore for arbitrary $(z_1,z_2) \in \tfrac{1}{N}D^d$ and a fixed $i \in \{1,\ldots, d\}$ it holds
$$\Prob(\widetilde{p}_1^{(i)} \in [z_1^{(i)}, z_1^{(i)} + \tfrac{1}{N}), \widetilde{p}_2^{(i)} \in [z_2^{(i)}, z_2^{(i)} + \tfrac{1}{N})) = \frac{1}{N(N-1)}.$$
Since $\widetilde{p}_1, \ldots, \widetilde{p}_N,$ are also jittered independently in the intervals of volume $\tfrac{1}{N^d}$ the claim follows.
\end{proof}

\begin{remark}
Let $d \geq 2, N \geq 5, \mathcal{P} = (p_j)_{j = 1}^N$ be a RSJ rank-1 lattice and $\widetilde{\mathcal{P}} = (\widetilde{p}_j)_{j = 1}^N$ be a LHS in $[0,1)^d.$ If $t \geq 3,$ then the distributions of $(p_j)_{j = 1}^t$ and $(\widetilde{p}_j)_{j = 1}^t$ differ.

To see this consider the discrete model. We will show that given $a,b \in \tfrac{1}{N} D^d$ there exists exactly one point set $X,$ consisting of $N$ points and corresponding to a RSJ rank-1 lattice such that $a,b \in X,$ but there are $[(N-2)!]^{d-1}$ such point sets corresponding to LHS.

The statement about LHS is obvious, so we focus on the point set corresponding to RSJ rank-1 lattice. The existence of $X$ follows by taking the shift $a$ and the generating vector $(b-a) \modulo 1.$ To see uniqueness recall that a rank-1 lattice is a cyclic subgroup of $\mathbb{T}^d,$ therefore any difference of two distinct elements is a generator of the lattice and determines it uniquely. This means $(b-a) \modulo 1$ is a generator of the underlying lattice $L$ and $X = (L + a) \modulo 1. $
\end{remark}

\begin{theorem}\label{PairNegDep}
Let $(p_j)_{j = 1}^N$ be a RSJ rank-1 lattice (in that case let $N$ be prime) or a Latin hypercube sampling in $[0,1)^d$ and $Q,R \in \mathcal{C}^d_1.$ Then
$$\Prob(p_1 \in Q, p_2 \in R) \leq \Prob(p_1 \in Q) \Prob(p_2 \in R),$$
meaning that RSJ rank-1 lattices and Latin hypercube sampling are pairwise negatively dependent. Moreover, RSJ rank-1 lattice and LHS are coordinatewise independent NQD sampling schemes.
\end{theorem}

\begin{proof}
We prove pairwise negative dependence first.
Put $Q = \prod_{i = 1}^d [q^{(i)}, 1), R = \prod_{i = 1}^d [r^{(i)},1).$ Due to Corollary \ref{JointUnifDistCor} and Lemma \ref{SSSpND}
\begin{align*}
& \Prob(p_1 \in Q, p_2 \in R) = \prod_{i = 1}^d  \Prob(p_1^{(i)} \in [q^{(i)},1), p_2^{(i)} \in [r^{(i)}, 1))
\\
& \leq \prod_{i = 1}^s  (1-q^{(i)})(1-r^{(i)}) \leq \Prob(p_1 \in Q) \Prob(p_2 \in R).
\end{align*}
That RSJ rank-1 lattice and LHS are coordinatewise independent NQD sampling schemes follows from  Corollary \ref{JointUnifDistCor} in conjunction with Lemma \ref{SSSpND}.
\end{proof}
Theorem \ref{PairNegDep} and \cite{Lem17} imply the following Corollary.
\begin{corollary}
Let $d,N \in \mathbb{N}$ and let $f:[0,1)^d \rightarrow \mathbb{R}$ be monotone in each coordinate. Denote by $\mu_{\mathcal{P}}$ a randomized QMC quadrature based on RSJ rank-1 lattice or LHS, using $N$ integration nodes and by $\mu^{MC}$ Monte Carlo quadrature using $N$ integration nodes. It holds
$$\Var(\mu_{\mathcal{P}}f) \leq \Var(\mu^{MC}f).$$
\end{corollary}
\begin{proof}
The claim follows by from \cite[Corollary 2]{Lem17} and the fact that RSJ rank-1 lattice and LHS are coordinatewise independent NQD sampling schemes.
\end{proof}
\subsection{Minimal Randomness for Randomly Shifted and Jittered Rank-1 Lattices}
In this section let $d \geq 2,$ and let $N \geq 5$ be a prime number. We want to argue that the randomization of rank-1 lattices proposed by us is in a way the minimal one leading to a pairwise negatively dependent sampling scheme. More precisely, we show that resigning from any step of the randomization (the random choice of the generating vector, the random uniform shift or the independent jittering) infringes either pairwise negative dependence or the sampling scheme property.

\begin{itemize}
\item[(i)]
First note that without the uniform shift we do not get a sampling scheme at all. Indeed, we have then $\Prob(p_1 \in [0,\tfrac{1}{N})^d) = \tfrac{1}{N}.$

\item[(ii)]
Now consider a situation in which we just shift all the points of the rank-1 lattice (possibly generated by a random vector) by a uniformly chosen vector on the torus. Then obviously the distances between the points on the torus remain unchanged.
Consider the distance function $\dist: [0,1)^2 \rightarrow [0,1)$ on the torus $\mathbb{T}^1$ given by
$$\dist(x,y):= \min \left( \max(x,y) - \min(x,y), 1 - \max(x,y) + \min(x,y) \right).$$
\begin{proposition}\label{fixedDistance}
Let $\mathcal{P} = (p_j)_{j = 1}^N$ be a sampling scheme such that for some $i = 1,\ldots, d,$ there exist a constant $0 <  \epsilon \leq \tfrac{1}{2}$ with
$$\Prob(\dist(p_1^{(i)}, p_2^{(i)}) \leq \epsilon) = 0.$$
Then $\mathcal{P}$ is not pairwise negatively dependent.
\end{proposition}
\begin{proof}
Without loss of generality let $i = d.$
Consider $Q := [0,1)^{d-1} \times [\tfrac{\epsilon}{2}, 1)$ and $R := [0,1)^{d-1} \times [1 - \tfrac{\epsilon}{2},1).$ We claim that
\begin{equation}\label{condProbClaim}
\Prob(p_1 \in Q | p_2 \in R) = 1,
\end{equation}
which already implies the statement of the proposition. Let  $p_2 \in R.$ Since almost surely $\dist(p_1^{(d)}, p_2^{(d)}) > \epsilon,$ we have $\Prob(p_1^{(d)} > \tfrac{\epsilon}{2}) = 1$ and so $\Prob(p_1 \in Q) = 1.$
\end{proof}
Proposition \ref{fixedDistance} shows that resigning from jittering we do not get a pairwise negatively dependent sampling scheme. As a side note, it also implies that lattice sampling, the earlier variant of LHS proposed by Patterson in \cite{Pat54}, does not provide a pairwise negatively dependent sampling scheme.
\item[(iii)]
 Finally consider the analogous construction for a fixed generating vector. To see that it is in general not pairwise negatively dependent take $N = 5, g = (\tfrac{1}{5}, \tfrac{1}{5})$ and $Q = [\tfrac{3}{5},1)^2, R = [\tfrac{4}{5},1)^2.$ Simple calculations reveal that in this case $$\Prob(p_1 \in Q, p_2 \in R) = \tfrac{1}{5} \tfrac{1}{2\binom{5}{2}} = \tfrac{1}{100}$$ and $$\Prob(p_1 \in Q)\Prob(p_2 \in R) = \tfrac{4}{625} < \tfrac{1}{100}.$$
\end{itemize}


\begin{thebibliography}{10}

\bibitem{DKS13}
{\sc J.~Dick, F.~Y. Kuo, and I.~H. Sloan}, {\em High dimensional integration --
  the quasi-{M}onte {C}arlo way}, Acta Numerica, 22 (2013), pp.~133--288.

\bibitem{DP10}
{\sc J.~Dick and F.~Pillichshammer}, {\em Digital Nets and Sequences},
  Cambridge University Press, Cambridge, 2010.

\bibitem{DDG18}
{\sc B.~Doerr, C.~Doerr, and M.~Gnewuch}, {\em Probabilistic lower discrepancy
  bounds for {L}atin hypercube samples}, in Contemporary Computational
  Mathematics -- a Celebration of the 80th Birthday of Ian Sloan, J.~Dick,
  F.~Y. Kuo, and H.~Wo\'zniakowski, eds., Springer-Verlag, 2018, pp.~339--350.

\bibitem{GH18}
{\sc M.~Gnewuch and N.~Hebbinghaus}, {\em Discrepancy bounds for a class of
  negatively dependent random points including {L}atin hypercube samples},
  2018.
\newblock Preprint.

\bibitem{Heb12}
{\sc N.~Hebbinghaus}, {\em Mixed sequences and application to multilevel
  algorithms}, master's thesis, Christ Church, University of Oxford, 2012.

\bibitem{Lem09}
{\sc C.~Lemieux}, {\em Monte Carlo and Quasi-Monte Carlo Sampling}, Springer,
  New York, 2009.

\bibitem{Lem17}
\leavevmode\vrule height 2pt depth -1.6pt width 23pt, {\em Negative dependence,
  scrambled nets, and variance bounds}, Mathematics of {O}perations {R}esearch,
  43 (2018), pp.~228--251.

\bibitem{MBC79}
{\sc M.~D. McKay, R.~J. Beckman, and W.~J. Conover}, {\em A comparison of three
  methods for selecting values of input variables in the analysis of output
  from a computer code}, Technometrics, 21 (1979), pp.~239--245.

\bibitem{Owe95}
{\sc A.~B. Owen}, {\em Randomly permuted $(t,m,s)$-nets and $(t,s)$-sequences},
  in Monte Carlo and Quasi-Monte Carlo Methods in Scientific Computing,
  H.~Niederreiter and P.~J.-S. Shiue, eds., New York, 1995, Springer,
  pp.~299--317.

\bibitem{Owe97}
\leavevmode\vrule height 2pt depth -1.6pt width 23pt, {\em {M}onte {C}arlo
  variance of scrambled equidistribution quadrature}, SIAM J. Numer. Anal., 34
  (1997), pp.~1884--1910.

\bibitem{Pat54}
{\sc H.~D. Patterson}, {\em The errors of lattice sampling}, J. Royal
  Statistical Society, Series B, 16 (1954), pp.~140--149.

\bibitem{SW04}
{\sc I.~H. Sloan and H.~Wo\'zniakowski}, {\em When does {M}onte {C}arlo depend
  polynomially on the number of variables?}, in Monte Carlo and Quasi-Monte
  Carlo Methods 2002, H.~Niederreiter, ed., Berlin-Heidelberg, 2004, Springer,
  pp.~407--437.

\bibitem{WGH19}
{\sc M.~Wnuk, M.~Gnewuch, and N.~Hebbinghaus}, {\em On negatively dependent
  sampling and probabilistic upper discrepancy bounds}.
\newblock Preprint, 2019.

\end{thebibliography}
\end{document}